\documentclass[a4paper]{article}

\usepackage[utf8]{inputenc}
\usepackage{lmodern}
\usepackage{array}
\usepackage{graphicx}
\usepackage{amssymb, amsfonts, amsthm, amsmath}
\usepackage[english]{babel}
\usepackage[pdfborder={0 0 0}]{hyperref}
\usepackage{tikz}
\usetikzlibrary{decorations}
\usetikzlibrary{decorations.pathreplacing}
\tikzset{individu/.style={draw,thick}}

\usepackage{subfigure}

\theoremstyle{plain}
\newtheorem{theorem}{Theorem}[section]

\newtheorem{fact}[theorem]{Fact}

\theoremstyle{definition}

\theoremstyle{remark}

\newcommand{\N}{\mathbb{N}}

\newcommand{\R}{\mathbb{R}}

\newcommand{\ind}[1]{\mathbf{1}_{\left\{#1\right\}}}

\numberwithin{equation}{section}

\DeclareMathOperator{\E}{\mathbb{E}}

\renewcommand{\P}{\mathbb{P}}

\newcommand{\calF}{\mathcal{F}}

\title{A short proof of the asymptotic of the minimum of the branching random walk after time $n$}
\author{Bastien Mallein}
\date{\today}

\newcommand{\T}{\mathbf{T}}

\renewcommand{\rho}{\varrho}
\renewcommand{\epsilon}{\varepsilon}

\begin{document}

\maketitle

\begin{abstract}
We write $R_n$ for the minimal position attained after time $n$ by a branching random walk in the boundary case. In this article, we prove that $R_n - \frac{1}{2} \log n$ converges in law toward a shifted Gumbel distribution.
\end{abstract}

\section{Introduction}

A branching random walk on $\R$ is particle system defined as follows. It starts with a unique individual located at position $0$ at time $0$. At each time $n \in \N$, every individual currently alive in the process dies, giving birth to children that are positioned around their parent according to i.i.d. versions of a point process. We write $\T$ for the genealogical tree of the process. For any $u \in \T$, we denote by $V(u)$ the position of individual $u$ and $|u|$ the generation to which $u$ belongs. The branching random walk is the random marked tree $(\T,V)$. In this article, we take interest in the asymptotic behaviour of the smallest position reached after time $n$, defined by
\begin{equation}
  \label{eqn:defrn}
  R_n = \inf_{u \in \T, |u| \geq n} V(u).
\end{equation}

We assume that the branching random walk is in the boundary case:
\begin{equation}
  \label{eqn:boundary}
  \E\left( \sum_{|u|=1} 1 \right) > 1, \text{ } \E\left( \sum_{|u|=1} e^{-V(u)} \right) = 1 \text{ and }  \E\left( \sum_{|u|=1} V(u)e^{-V(u)} \right) = 0,
\end{equation}
and that the displacement law is non-lattice. It is well-known (see e.g. the discussions in \cite{BeG11,Jaf11}) that under mild integrability assumptions, a branching random walk can be mapped with a branching random walk in the boundary case by an affine transformation. We also assume the following, classical, integrability assumptions on the reproduction law of the process
\begin{equation}
  \label{eqn:variance}
  \sigma^2 := \E\left( \sum_{|u|=1} V(u)^2 e^{-V(u)} \right) \in (0,+\infty)
\end{equation}
\begin{equation}
  \label{eqn:integrability}
  \E\left( \sum_{|u|=1} e^{-V(u)} \log_+ \left( \sum_{|u|=1} (1 + V(u)_+)e^{-V(u)} \right)^2 \right) < +\infty.
\end{equation}
These conditions replace, in some sense, the $L \log L$ integrability condition for the Galton-Watson process.

Under assumption \eqref{eqn:boundary}, the Galton-Watson tree $\T$ is supercritical, and the survival event $\{ \# \T = +\infty\}$ occurs with positive probability. For any $n \in \N$, we introduce
\begin{equation}
  \label{eqn:martingales}
  W_n = \sum_{|u|=n} e^{-V(u)} \quad \text{and} \quad Z_n = \sum_{|u|=n} V(u)e^{-V(u)}.
\end{equation}
By \eqref{eqn:boundary} and the branching property of the branching random walk, $(W_n)$ and $(Z_n)$ are martingales. They are respectively called the additive martingale and the derivative martingale. Under assumptions \eqref{eqn:variance} and \eqref{eqn:integrability}, there exists a random variable $Z_\infty$ which is a.s. positive on $S$ such that
\begin{equation}
  \label{eqn:defZinfty}
  \lim_{n \to +\infty} W_n = 0 \quad \text{and} \quad \lim_{n \to +\infty} Z_n = Z_\infty \quad \text{a.s.}
\end{equation}
This result was proved in \cite{BiK04} under stronger integrability assumptions, then extended in \cite{Aid13}. It is proved in \cite{Che15} that assumption \eqref{eqn:integrability} is necessary and sufficient for \eqref{eqn:defZinfty} to hold. Aïdékon and Shi \cite{AiS10} obtained the precise asymptotic behaviour of $W_n$. This result is recalled in Fact \ref{fct:ais}.

The asymptotic behaviour of the extremal individuals in the branching random walk has been an object of interest in the recent years. We denote by $M_n = \min_{|u|=n} V(u)$ the minimal displacement at time $n$. Hammersley \cite{Ham74}, Kingman \cite{Kin75} and Biggins \cite{Big76} proved that $\frac{M_n}{n}$ converges almost surely to 0. The second order has been obtained independently by Addario-Berry and Reed \cite{ABR09}, who proved that $M_n - \frac{3}{2} \log n$ is tight, and by Hu and Shi \cite{HuS09}, who obtained that $\frac{M_n}{\log n}$ converges in probability toward $\frac{3}{2}$, while experiencing almost sure fluctuations. Finally, Aïdékon \cite{Aid13} completed the picture by proving the convergence in law of $M_n - \frac{3}{2} \log n$. A more precise statement of this result is given in Fact \ref{fct:aid}.

In this article, we study the asymptotic behaviour of $R_n = \min_{k \geq n} M_k$, the lowest position reached after time $n$. Using previous sharp estimates on the branching random walks from \cite{Aid13,AiS10,Mad16}, we compute the joint convergence in law of $Z_n$, $M_n$ and $R_n$.
\begin{theorem}
\label{thm:main}
Under assumptions \eqref{eqn:boundary}, \eqref{eqn:variance} and \eqref{eqn:integrability}, we have
\[
  \lim_{n \to +\infty} \left(Z_n,M_n-\tfrac{3}{2} \log n, R_n - \tfrac{1}{2}\log n\right) = (Z_\infty, W,L) \quad \text{in law,}
\]
where $\P(W \geq x,L \geq y|Z_\infty) = \exp\left( - c_* Z_\infty e^x - c' Z_\infty e^y \right)$, $c_*$ is the constant defined in Fact \ref{fct:aid}, $c_M$ the constant defined in Fact \ref{fct:mad} and $c' = \sqrt{\frac{2}{\pi \sigma^2}} c_M$.
\end{theorem}

As a side result, we observe that
\begin{equation}
  \label{eqn:firstorder}
  \lim_{n \to +\infty} \frac{R_n}{\log n} = \frac{1}{2} \quad \text{a.s. on the event} \quad \{ \# \T = +\infty\}.
\end{equation}
In particular, we observe that $\frac{R_n}{\log n}$ does not fluctuates almost surely at scale $n$, contrarily to $M_n$ the smallest displacement at time $n$ (see \cite{HuS09}).

\section{Proof of Theorem \ref{thm:main}}

In a first time, we recall some branching random walk estimates that we use to prove Theorem \ref{thm:main}. We first give a precise statement for the convergence in law of $M_n - \frac{3}{2} \log n$.
\begin{fact}[Aïdékon \cite{Aid13}]
\label{fct:aid}
Under assumptions \eqref{eqn:boundary}, \eqref{eqn:variance} and \eqref{eqn:integrability} we have
\[
  \lim_{n \to +\infty} \left(Z_n,M_n-\tfrac{3}{2} \log n\right) = (Z_\infty, W) \quad \text{in law,}
\]
where $\P(W \geq x|Z_\infty) = \exp\left( - c_* Z_\infty e^x \right)$ and $c_*$ is a fixed constant, that depends only on the reproduction law of the branching random walk.
\end{fact}
More precisely, in \cite[Theorem 1.1]{Aid13} we have only the convergence in law of $M_n - \frac{3}{2} \log n$ , but the joint convergence of $(Z_n,M_n)$ follows immediately from the proof. This result can also be obtained as a straightforward consequence of \cite[Theorem 1.1]{Mad15}.

By \eqref{eqn:defZinfty}, we have $\lim_{n \to +\infty} W_n = 0$ a.s. Aïdékon and Shi \cite{AiS10} computed the rate of decay for this quantity.
\begin{fact}[Aïdékon and Shi \cite{AiS10}]
\label{fct:ais}
Under assumptions \eqref{eqn:boundary}, \eqref{eqn:variance} and \eqref{eqn:integrability}, we have
\[
  \lim_{n \to +\infty} n^{1/2} W_n = \sqrt{\frac{2}{\pi \sigma^2}} Z_\infty \quad \text{in probability.}
\]
\end{fact}

Finally, Madaule \cite{Mad16} obtained a precise estimate for the left tail asymptotic of $R_0$ the minimal position attained by the branching random walk.
\begin{fact}[Madaule \cite{Mad16}]
\label{fct:mad}
Under assumptions \eqref{eqn:boundary}, \eqref{eqn:variance} and \eqref{eqn:integrability}, there exists $c_M>0$ such that $\lim_{x \to +\infty} e^{x} \P(R_0 \leq -x) = c_M$.
\end{fact}

Using these two additional results, we are able to compute the joint asymptotic behaviour of $Z_n, M_n$ and $R_n$.
\begin{proof}[Proof of Theorem \ref{thm:main}]
For any $n \in \N$, we denote by $\calF_n = \sigma(u,V(u), |u| \leq n)$. For any $x,y \in \R$, we write
\[ r_n(x) = \frac{1}{2} \log n + x \quad \text{and} \quad m_n(y) = \frac{3}{2} \log n + y.\]
For any $u,v \in \T$, we write $u<v$ if $u$ is an ancestor of $v$. We observe that
\[
  R_n = \min_{|u| \geq n} V(u) = \min_{|u|=n} \left( V(u) + \min_{v>u} V(v)-V(u) \right),
\]
where, by the branching property $(\min_{v> u} V(v)-V(u), |u|=n)$ are i.i.d. copies of $R_0$ that are independent with $\calF_n$. Consequently, we have
\[
  \P\left( R_n \geq r_n(x), M_n \geq m_n(y) \middle| \calF_n \right) = \ind{M_n \geq m_n(y)} \prod_{|u|=n} \phi(r_n(x)-V(u)) \quad \text{a.s.},
\]
where we set $\phi(z) = \P(R_0 \geq z)$. By Fact \ref{fct:mad}, for any $\epsilon>0$, there exists $A>0$ such that for any $z \leq - A$, we have
\[
 1- (c_M + \epsilon) e^z \leq \phi(z) \leq 1 - (c_M - \epsilon) e^z.
\]

Consequently, for any $n \geq e^{y-x+A}$, we have a.s. on the event $\{M_n \geq m_n(y)\}$,
\begin{multline*}
  \P\left( R_n \geq r_n(x), M_n \geq m_n(y) \middle| \calF_n \right) \geq \prod_{|u|=n} \left( 1 - (c_M+\epsilon)e^{r_n(x)-V(u)}\right)\\
  \geq \exp\left( - \sum_{|u|=n} (c_M + \epsilon)e^{r_n(x)-V(u)}\right)
  = \exp\left( - (c_M+\epsilon)e^x n^{1/2}W_n \right).
\end{multline*}
Moreover, there exists $\delta>0$ such that $(1-h) \leq e^{-(1-\epsilon)h}$ for any $0 \leq h < \delta$. Therefore, a.s. on the event $\{M_n \geq m_n(y)\}$ we have for any $n$ large enough,
\begin{align*}
  \P\left( R_n \geq r_n(x), M_n \geq m_n(y) \middle| \calF_n \right) &\leq \prod_{|u|=n} \left( 1 - (c_M-\epsilon)e^{r_n(x)-V(u)}\right)\\
  &\leq \exp\left( - (1-\epsilon)(c_M+\epsilon)e^x n^{1/2}W_n \right).
\end{align*}

By Fact \ref{fct:ais} and Slutsky's lemma, we can extend Fact \ref{fct:aid} into
\[
  \lim_{n \to +\infty} (n^{1/2}W_n, Z_n, M_n - m_n(0)) = \left( \sqrt{\frac{2}{\pi \sigma^2}} Z_\infty, Z_\infty,W \right)
\]
As a consequence, for any continuous bounded function $\phi$, letting $n \to +\infty$ then $\epsilon \to 0$, we have
\[
  \lim_{n \to +\infty} \E\left( \phi(Z_n) \ind{M_n \geq m_n(y),R_n \geq r_n(x)} \right) = \E\left( \phi(Z_\infty) \ind{W \geq y} e^{- c' e^x Z_\infty} \right),
\]
which concludes the proof.
\end{proof}

In a second time, we prove \eqref{eqn:firstorder}. For any $n \in \N$, we write
\[
  \tau_n = \inf\{ k \geq n : M_k = R_n\}.
\]
We observe easily that $\frac{R_n}{\log n} \geq \frac{M_{\tau_n}}{\log n} \geq \frac{M_{\tau_n}}{\log \tau_n}$. Thus by \cite[Theorem 1.2]{HuS09}, we have $\liminf_{n \to +\infty} \frac{R_n}{\log n} \geq \liminf_{n \to +\infty} \frac{M_n}{n} = \frac{1}{2}$ a.s.

To prove the upper bound, we use \cite[Lemma 3.5]{Hu09}. There exists $c>0$ and $K>0$ such that for any $n \in \N$ and $0 \leq \lambda \leq \frac{\log n}{3}$, we have
\[
  \P\left( \exists u \in \T : |u| \in [n,2n], V(u) - \tfrac{1}{2} \log n + \lambda \in [0,K] \right) \geq ce^{-\lambda}.
\]
In particular $\P(R_n \leq \tfrac{1}{2} \log n + K) \geq c$. We conclude with a cutting argument.

By \cite[Lemma 2.4]{Mal16}, there exists $a>0$ and $\rho>1$ such that almost surely on $\{\#\T=+\infty\}$, for any $k$ large enough we have
\[
  \# \{ u \in \T : |u|=k, V(u) \leq k a \} \geq \rho^k.
\]
As each individual $u$ alive at time $k$ starts an independent branching random walk from position $V(u)$, for any $\epsilon>0$, we have almost surely on $\{\#\T = +\infty\}$, for all $n$ large enough,
\[
  \P\left( R_{n + \epsilon \log n} \geq (\tfrac{1}{2} + a\epsilon) \log n + K \middle| \calF_{\epsilon \log n} \right) \leq (1-c)^{\rho^{\epsilon \log n}}.
\]
We conclude by Borel-Cantelli lemma that $\limsup_{n \to +\infty} \frac{R_n}{\log n} \leq \frac{1}{2}$ a.s, which completes the proof of \eqref{eqn:firstorder}.

\bibliographystyle{alpha}

\end{document}